\documentclass[11pt]{amsart}
\usepackage[hmargin=3cm,vmargin=3.3cm]{geometry}
\usepackage{amssymb}
\usepackage{physics}

\newtheorem{theorem}{Theorem}
\newtheorem{lemma}{Lemma}
\theoremstyle{remark}
\newtheorem{remark}{Remark}

\newcommand\RR{\mathbb{R}}
\newcommand\CC{\mathbb{C}}
\newcommand\cC{\mathcal{C}}
\newcommand\cE{\mathcal{E}}
\newcommand\cH{\mathcal{H}}
\newcommand\cI{\mathcal{I}}
\newcommand\cQ{\mathcal{Q}}
\newcommand\cR{\mathcal{R}}

\newcommand\cN{\mathcal{N}}
\newcommand\ep{\varepsilon}
\newcommand\ii{\textnormal{i}}

\renewcommand{\Re}{\mathfrak{R}}

\title[Liouville theorem for 1D Gross-Pitaevskii equation]
{A new proof of a Liouville theorem for the one dimensional Gross-Pitaevskii equation}

\author{Micha{\l} Kowalczyk}
\address{Departamento de Ingenier\'{\i}a Matem\'atica and Centro
de Modelamiento Matem\'atico (UMI 2807 CNRS), Universidad de Chile, Casilla
170 Correo 3, Santiago, Chile.}
\email {kowalczy@dim.uchile.cl}

\author{Yvan Martel}
\address{Laboratoire de mathématiques de Versailles,
UVSQ, Université Paris-Saclay, CNRS, and Institut Universitaire de France,
45 avenue des États-Unis,
78035 Versailles Cedex, France}
\email{yvan.martel@uvsq.fr}

\thanks{M.K. was partially funded by Chilean research grants FONDECYT 1250156 and ANID project FB210005.
Part of this work was done while Y.M. was visiting CMM, Universidad de Chile}

\subjclass[2010]{35L71 (primary), 35B40, 37K40}

\begin{document}
\begin{abstract}
The asymptotic stability of the black and dark solitons of the one-dimensional 
Gross-Pitaevskii equation was proved by Béthuel, Gravejat and Smets~\cite{BGS1}
and Gravejat and Smets~\cite{GS15}, using a rigidity property 
in the vicinity of solitons.
We provide an alternate proof of the Liouville theorems in \cite{BGS1,GS15} using a factorization
identity for the linearized operator which trivializes the spectral analysis.
\end{abstract}

\maketitle

\section{Introduction}

We consider the one-dimensional Gross-Pitaevskii equation
\begin{equation*}
\ii\partial_t\psi+\partial_x^2 \psi+\psi (1-|\psi|^2)=0,\quad (t,x)\in \RR^2
\tag{GP}
\end{equation*}
for a function $\psi:(t,x)\in\RR^2\mapsto \psi(t,x)\in \CC$, with the condition
$\lim_{|x|\to\infty} |\psi(t,x)|=1$.
For a solution of (GP), the Hamiltonian is formally conserved
\[
E (\psi)=\frac12\int(\partial_x\psi)^2+\frac14\int(1-|\psi|^2)^2.
\]
Using the notation $\eta = 1-|\psi|^2$, it is natural to define the energy space
as follows
\[
\cE = \{ \psi \in \cC( \RR ; \CC ): \psi'\in L^2(\RR) \mbox{ and } \eta \in L^2(\RR)\}.
\]
We denote
\[
\|f\|:=\|f\|_{L^2},\quad \|f\|_\rho=\|\rho^\frac12 f\|,
\quad \|f\|_\cH=\left(\|f'\|^2 + \|f\|_\rho^2\right)^\frac12
,\quad \rho(x)=\sech(x).
\]
Following~\cite{GS15}, we equip the energy space $\cE$ with the distance
\[
d(\psi_1,\psi_2)= \left(\|\psi_1'-\psi_2'\|_\cH^2+\|\eta_1-\eta_2\|^2\right)^\frac12
\]
so that $(\cE,d)$ is a complete metric space. Recall from~\cite{Ge08,GS15,Zh01} that the Cauchy problem
is globally well-posed in $\cE$: for any $\psi_0\in \cE$, there exists
a unique global solution $\psi\in \cC(\RR,\cE)$ of (GP) with $\psi(0)=\psi_0$.

It is well-known that for any velocity $c\in (-\sqrt{2},\sqrt{2})$, there exists a nontrivial traveling wave solution $\psi(t,x)=U_c(t-cx)$ to this problem where $U_c(x)$ is the solution of
\[
-\ii cU_c'+U_c''+U_c(1-|U_c|^2)=0 \quad \hbox{on $\RR$}
\]
explicitly given by the formula
\[
U_c=R_c+\ii I_c,\quad 
R_c(x)=\sqrt{\frac{2-c^2}{2}} \tanh\left(\frac{\sqrt{2-c^2}}{2} x\right),\quad
I_c= \frac{c}{\sqrt{2}}.
\]
In case $c\neq 0$, the traveling wave solution is called dark soliton and in case $c=0$, it is called black soliton. 
The orbital stability of both kinds of solitons was proved in a satisfactory
functional setting; we refer to \cite[Theorem 1]{BGS1} and \cite[Theorem 1]{GS15}.

We are interested in the question of asymptotic stability of the family of traveling waves in the framework developed 
by Béthuel, Gravejat, Smets~\cite{BGS1} and Gravejat, Smets~\cite{GS15}.
From those articles, their approach relies on the following Liouville theorem for smooth solutions of (GP) that are close 
to a traveling wave and uniformly localized in space.

\begin{theorem}[\cite{BGS1,GS15}]\label{th:1}
Let $c_0\in (-\sqrt{2}, \sqrt{2})$. Let $M>0$ and $\gamma>0$.
There exists $\alpha_0>0$ such that if
a solution $\psi\in \cC(\RR;\cE)$ of \textnormal{(GP)}
satisfies $\psi\in\cC^\infty(\RR\times\RR)$, $d(\psi(0),U_{c_0})\leq \alpha_0$ and
\[
\inf_{a\in \RR} \biggl\{ |\eta(t, x+a)|
+\sum_{k=1,2,3} |\partial_x^k \psi(t, x+a)|\biggr\}\leq M e^{-\gamma|x|}
\]
then there exist $c_1\in (-\sqrt{2},\sqrt{2})$ and $a_1\in\RR$ such that
$\psi(t,x)=U_{c_1}(x-c_1 t+a_1)$ on $\RR\times\RR$. 
\end{theorem}
In this Note, we provide an alternate proof of Theorem~\ref{th:1},
inspired by recent articles on the asymptotic stability of solitons for
the nonlinear Klein-Gordon equation~\cite{KoMM} and the nonlinear Schrödinger equation~\cite{Ma22}.
As in those articles, we introduce a factorization identity of the linearized operator
around a traveling wave which leads to a transformed problem
with trivial spectral properties. 
We refer to Lemma~\ref{le:02} for the identity and to Remark~\ref{rk:2}
for heuristics.
The identity obtained in Lemma~\ref{le:02} is certainly related to the integrability of the model.
However, as shown in~\cite{Ma22,Ri24} for nonlinear Schrödinger models, 
it is possible to extend the arguments to less specific situations.
Therefore, we hope that the factorization will be useful elsewhere. 
Finally, we expect that our approach will enable a direct 
proof of the asymptotic stability for any $c\in (-\sqrt{2},\sqrt{2})$, as in 
\cite{KoMM,KoMa,Ma22}.

\section{Stability and modulation}

For $c\in (-\sqrt{2},\sqrt{2})$, we set
$\beta=\sqrt{2-c^2}>0$ and we define
\[
Q_c=1-|U_c|^2 =\sqrt{2} R_c'=\frac{\beta^2}2 \sech^2\left(\frac{\beta x}2\right),
\quad Q_c''-\beta^2 Q_c +3Q_c^2=0.
\]


\begin{lemma}\label{le:01}
Under the assumptions of Theorem~\ref{th:1}, for $\alpha_0>0$ sufficiently small, there exist
 functions $a\in \cC^\infty(\RR,\RR)$, $c\in \cC^\infty(\RR,(-\sqrt2,\sqrt2))$, $\theta\in\cC^\infty(\RR,\RR)$
such that
\[
\psi(t,x)= e^{\ii\theta(t)}\left(U_{c(t)}(x-a(t))+\ep(t,x-a(t))\right)
\]
where $\ep\in\cC^\infty(\RR\times\RR)$ satisfies the orthogonality conditions
\[
\int \Re(U'_c \bar \ep)=\int \Re (\ii U'_c \bar \ep)=\int \Re(\ii R_cQ_c \bar \ep)=0
\]
and the equation
\[
\ii \partial_t\ep +\partial_x^2\ep -\ii c\partial_x \ep+q=\Omega
\]
with
\begin{align*}
q&=(1-|U_c+\ep|^2)(U_c+\ep)-(1-|U_c|^2) U_c,\\
\Omega&=\ii (\dot a-c) (U_c'+\partial_x\ep)-\ii \dot c\partial_c U_c + \dot \theta(U_c+\ep).
\end{align*}
Moreover, setting $\zeta(t,x)=Q_c(x)-\eta(t,x+a(t))$,
there exists $\gamma\in (0,\beta/2)$ such that, for all $t\in\RR$,
\begin{align*}
&\|\ep(t)\|_\cH+\|\zeta(t)\|+|c(t)-c_0|\lesssim \alpha_0,\\
&|\dot a(t)- c(t) |^2+|\dot c(t)|+|\dot\theta(t)|^2 \lesssim \int \rho^\gamma|\ep(t)|^2,\\
&|\zeta(t)|+\sum_{k=1,2,3} |\partial_x^k \ep(t)| \lesssim \rho^\gamma.
\end{align*}
\end{lemma}
We shall not reproduce here the proof of Lemma~\ref{le:01}. See
\cite{BGS1,GS15}.
The bound on the time derivative of the parameters $a$, $c$ and $\theta$
is deduced from the proofs in~\cite{BGS1,GS15}. Note that only $|\dot c|$ is quadratic in
$\ep$.

We denote $\ep=\ep_1+\ii \ep_2$, $q=q_1+\ii q_2$ and 
$\Omega=\Omega_1+\ii \Omega_2$. 
The orthogonality conditions become
\begin{equation*}
\int Q_c \ep_1= \int Q_c \ep_2 = \int R_c Q_c \ep_2=0.
\end{equation*}
We observe from its definition $\zeta=Q_c - (1-|U_c+\ep|^2)$ that
\[
\zeta= Q_c -( 1 - |U_c|^2 - 2\Re(U_c \ep)-|\ep|^2)
= 2R_c\ep_1 + c \sqrt{2} \ep_2 +|\ep|^2.
\]
It holds $q_1 =Q_c \ep_1-R_c\zeta -\zeta \ep_1$ and
$q_2 =Q_c \ep_2-\frac c{\sqrt{2}} \zeta-\zeta \ep_2$. Thus
\begin{align*}
\partial_t\ep_1&=-\partial_x^2 \ep_2+c\partial_x\ep_1-Q_c \ep_2+\frac{c}{\sqrt{2}}\zeta+\zeta \ep_2+\Omega_2\\
\partial_t\ep_2&=\partial_x^2 \ep_1+c\partial_x\ep_2+Q_c \ep_1-R_c\zeta -\zeta\ep_1 - \Omega_1
\end{align*}
Define
\begin{gather*}
 L_+ =-\partial^2_x +\beta^2 -3Q_c ,\quad
L_- =-\partial_x^2 +c^2 -Q_c, \\
 S_c =\partial_x+\sqrt{2} R_c = Q_c \cdot \partial_x \cdot Q_c^{-1}, 
\quad S_c^\star=-\partial_x+\sqrt{2} R_c
 =- Q_c^{-1}\cdot \partial_x \cdot Q_c.
\end{gather*}
Using $2R_c^2=\beta^2-2Q_c$, we obtain
\[
\Bigg\{\begin{aligned}
\partial_t\ep_1&=L_-\ep_2+c S_c\ep_1+N_2+\Omega_2\\
\partial_t\ep_2&=-L_+\ep_1-cS_c^\star\ep_2-N_1-\Omega_1
\end{aligned}
\]
where
\begin{align*}
&N_1=-R_c|\ep|^2-2 R_c \ep_1^2-\sqrt{2} c\ep_1\ep_2 -\ep_1|\ep|^2\\
&N_2=\frac{c}{\sqrt{2}} |\ep|^2+2 R_c\ep_1\ep_2+\sqrt{2} c \ep_2^2+\ep_2|\ep|^2
\end{align*}
and
\begin{align*}
&\Omega_1=-(\dot a -c)\partial_x\ep_2+\frac{\dot c}{\sqrt{2}}
+\dot\theta(R_c+\ep_2) \\
& 
\Omega_2 = (\dot a-c)(R_c'+\partial_x\ep_1)
-\dot c \partial_c R_c+\dot\theta\left(\frac{c}{\sqrt{2}}+\ep_2\right).
\end{align*}

\section{The transformed problem}
\subsection{Factorization}
We state the key lemma of this Note.
\begin{lemma}\label{le:02}
For any $c\in (-\sqrt{2},\sqrt{2})$, it holds
$
L_+= S_c^\star S_c$ and $S_c (L_- - c^2 ) S_c^\star = (\partial^2_x-\beta^2) \partial_x^2$.
\end{lemma}
\begin{remark}\label{rk:2}
To illustrate heuristically the interest of the above identities, observe that for $c=0$ (to simplify the 
exposition), setting $w_1=S_0 \ep_1$ and $\ep_2 = S_0^\star w_2$, one obtains
\[
\Bigg\{\begin{aligned}
\partial_t\ep_1&=L_-\ep_2 \\
\partial_t\ep_2&=-L_+\ep_1 
\end{aligned}
\quad \iff\quad
\Bigg\{\begin{aligned}
\partial_t w_1&= (\partial_{x}^2-2 )\partial_{x}^2 w_2\\
\partial_t w_2& =-w_1
\end{aligned}
\]
Being without any potential, the system for $(w_1,w_2)$ is simpler;
see Lemma~\ref{le:07}.
\end{remark}
\begin{proof}
First, the identity $L_+= S_c^\star S_c$ is standard and related to 
the fact that $L_+Q_c=0$.
Second, we recall that
$L_- - c^2=-\partial^2_x-\sqrt{2} R'_c$ and so, using $R_c'' = -\sqrt{2}R_c'R_c$,
\begin{align*}
(L_- - c^2) S_c^\star&=(-\partial_x^2-\sqrt{2} R'_c)(-\partial_x+\sqrt{2} R_c)\\
&=\partial_x^3-\sqrt{2}R_c\partial_x^2-2\sqrt 2 R'_c\partial_x-\sqrt 2 R''_c+\sqrt 2 R'_c\partial_x-2 R_c' R_c
=A \partial_x
\end{align*}
where $A=\partial_x^2-\sqrt 2 R_c\partial_x-\sqrt 2R_c'$.
Thus,
\[
S_c (L_- -c^2) S_c^\star=  S_c A\partial_x.
\]
We calculate, using
$2R_c^2=\beta^2-2\sqrt{2}R_c'$ and $R_c'' = -\sqrt{2}R_c'R_c$,
\begin{align*}
S_cA&=(\partial_x+\sqrt 2 R_c)(\partial_x^2-\sqrt 2 R_c\partial_x-\sqrt 2R_c')\\
&=\partial^3_x-2 \sqrt{2} R_c' \partial_x-\sqrt 2 R_c''
-2 R_c^2 \partial_x-2 R_c R_c'=(\partial^2_x-\beta^2) \partial_x
\end{align*}
which finishes the proof of the second identity.
\end{proof}

In view of the definition of $\ep_2$ in Remark~\ref{rk:2}, we will need to invert
the operator $S_c^\star$.

\begin{lemma}\label{le:03}
For any $f\in W^{1,\infty}(\RR)$ such that $\int f Q_c=0$, 
the function $g\in W^{2,\infty}(\RR)$ defined by
\[
g(x) = \frac 1{Q_c(x)} \int_x^\infty f Q_c = - \frac 1{Q_c(x)} \int_{-\infty}^x f Q_c
\]
satisfies $S_c^\star g=f$. Moveover,
\[
|g(x)|+|g'(x)|\lesssim \sup_{|y|\geq |x|} |f(y)|,
\quad|g'(x)|+|g''(x)|\lesssim \sup_{|y|\geq |x|} |f'(y)| + Q_c(x) |f(x)|
\]
and, for any $0<\kappa<2\beta$,
\[
\| \rho^\kappa g\| + \|\rho^\kappa g'\| 
\lesssim \|\rho^\frac\kappa2 f\|.
\]
\end{lemma}
\begin{remark}\label{rk:3}
If $S_c^\star g = 0$, then $Q_c g=C$, where $C$ is a constant.
Thus, for a given $f\in W^{1,\infty}(\RR)$, the function $g$ defined in Lemma~\ref{le:03} is the only bounded solution of $S_c^\star g = f$.
\end{remark}
\begin{proof}
Using $S_c^\star=-Q_c^{-1}\cdot \partial_x\cdot Q_c$, we check that $S_c^\star g=f$.
Moreover, 
the estimate $|g(x)|+|g'(x)|\lesssim \sup_{|y|\geq |x|} |f(y)|$ follows easily from the definition of $g$ and
\[
g' = -\frac{Q_c'}{Q_c^2}\int_x^\infty fQ_c-f.
\]
Integrating by parts, we also get for $x>0$,
\[
g'(x)= -\frac{Q_c'}{Q_c^2}\int_x^{\infty} f'(y) \int_y^\infty Q_c dy
+k_c(x) f(x)
\quad \mbox{where}\quad
k_c(x)= -\frac{Q_c'}{Q_c^2}\int_x^{\infty} Q_c -1.
\]
Replacing $Q_c$ by its explicit expression, we obtain
$k_c(x) = -e^{-\beta x}$. Proceeding similarly for $x<0$, we obtain the pointwise estimate on $g'$.
To obtain the pointwise estimate on $g''$, it suffices to differentiate the expression for $g'$
and to use similar estimates.

Now, we observe that for $0\lesssim \kappa<2\beta$ and for all $x>0$,
\begin{multline*}
\rho^{2\kappa} g^2
+\rho^{2\kappa} (g')^2\lesssim 
\frac{\rho^{2\kappa}}{Q_c^2}\left(\int_x^\infty f Q_c\right)^2
+\rho^{2\kappa}f^2\\
\lesssim \frac{\rho^{2\kappa}}{Q_c^2}\left(\int_x^\infty \rho^{-\kappa}Q_c^2\right) 
\left(\int \rho^{\kappa} f^2\right) +\rho^{2\kappa}f^2
\lesssim \rho^\kappa\left(\int \rho^\kappa f^2\right)+\rho^{2\kappa}f^2.
\end{multline*}
This, together with the analogue estimate for $x<0$, 
implies $\| \rho^\kappa g\| + \|\rho^\kappa g'\| 
\lesssim \|\rho^\frac\kappa2 f\|$ by integrating in $x$.
\end{proof}

\subsection{First change of variables}\label{S:3.2}
We set $v_1=S_c \ep_1$. Then, we use Lemma~\ref{le:03} and the orthogonality $\int \ep_2 Q_c=0$ to 
define a smooth and bounded function $w_2$ such that
$S_c^\star w_2 = \ep_2$.
We determine the equations for $v_1$ and $w_2$. First,
\[
\partial_t v_1=S_c L_- S_c^\star w_2+cS_c v_1+ P_2+\Theta_2,
\quad
P_2=S_c N_2+\sqrt{2} \dot c \ep_1 \partial_c R_c,
\quad \Theta_2 = S_c \Omega_2.
\]
From $S_c^\star w_2=\ep_2$ and $L_+=S_c^\star S_c$, we find
\begin{align*}
\partial_t S_c^\star w_2= S^\star_c \partial_t w_2+ \sqrt{2}\dot c w_2\partial_cR_c
&=-L_+\ep_1-cS^\star\ep_2-N_1-\Omega_1\\
&=-S_c^\star (v_1+c S_c^\star w_2)-N_1-\Omega_1,
\end{align*}
hence, rearranging
\[
S_c^\star\left(\partial_t w_2+v_1+c S_c^\star w_2\right)=-W -N_1-\Omega_1,
\quad W=\sqrt{2} \dot c w_2 \partial_c R_c.
\]
Since the left hand side of the above identity is orthogonal to $Q_c$, 
we have
\[
-W-N_1-\Omega_1
=-W^\perp-N_1^\perp-\Omega_1^\perp,
\]
where $f^\perp=f-\frac{Q_c}{\|Q_c\|^2}\int f Q_c $.
Moreover, by Lemma~\ref{le:03} we can define 
$P_1$ and $\Theta_1$ such that
\[
S_c^\star P_1=-W^\perp-N_1^\perp,\quad 
S_c^\star \Theta_1 = -\Omega_1^\perp.
\]
We obtain
\begin{align*}
\partial_t v_1&=S_c L_- S_c^\star w_2+cS_c v_1+ P_2+\Theta_2\\
\partial_t w_2&=-v_1-c S_c^\star w_2+P_1+\Theta_1
\end{align*}

\subsection{Second change of variables}
We define
\[
w_1=v_1+c S_c w_2+\frac{1}{\sqrt{2}} |\ep|^2.
\]
Since $S_c^\star w_2 = \ep_2$, we have
$S_c w_2 = \ep_2+ 2\partial_x w_2$
and thus, using $\zeta=2R_c\ep_1+c\sqrt{2}\ep_2+|\ep|^2$,
\[
w_1=v_1+c\ep_2+2c\partial_x w_2+\frac{1}{\sqrt{2}} |\ep|^2
=\partial_x \ep_1+2c\partial_x w_2+\frac{\zeta}{\sqrt{2}}.
\]
By Lemma \ref{le:01}, the last expression shows that the function $w_1$ has exponential decay at infinity, unlike the function $v_1$ (see Lemma \ref{le:06} below).

First, we compute using the definition of $w_1$ and then Lemma~\ref{le:03}
\begin{align*}
\partial_t w_1 &=\partial_t v_1+\dot c S_c w_2+\sqrt{2} c \dot c w_2\partial_c R_c +c S_c \partial_t w_2+\frac 1{\sqrt{2}} \partial_t |\ep|^2\\
&=S_cL_-S_c^\star w_2 - c^2 S_c S_c^\star w_2+F_2\\
&= (\partial_{x}^2-\beta^2 )\partial_{x}^2 w_2+F_2
\end{align*}
where 
\[
F_2=P_2+\Theta_2+\dot c S_c w_2+\sqrt{2} c\dot c w_2\partial_c R_c
+c S_c P_1+cS_c\Theta_1+\frac 1{\sqrt{2}} \partial_t |\ep|^2.
\]
Second, using $v_1=w_1-c S_c w_2-\frac 1{\sqrt{2}}|\ep|^2$ and $S_c-S_c^\star=2\partial_x$,
we obtain 
\[
\partial_t w_2=-w_1+2 c\partial_x w_2 + F_1\quad\mbox{where}\quad
F_1=P_1+\Theta_1+\frac 1{\sqrt{2}} |\ep|^2.
\]
Summarizing, the transformed problem is defined by
\[
\Bigg\{\begin{aligned}
\partial_t w_1&= (\partial_{x}^2-\beta^2 )\partial_{x}^2 w_2+F_2\\
\partial_t w_2& =-w_1+2 c\partial_x w_2 + F_1
\end{aligned}
\]

\section{Technical Lemmas}
\begin{lemma}\label{le:10}
For any $\kappa>0$ and any $f\in \cE$,
\[
\| f \rho^\kappa \|_{L^\infty} \lesssim \|f\|_\cH
\]
and for any $A>0$,
\[
\int_0^A f^2 \rho^\kappa \lesssim \int_0^1 f^2 + 
\left(\int_0^A (f')^2\right)^\frac 12 \left(\int_0^A f^2 \rho^{2\kappa}\right)^\frac 12.
\]
\end{lemma}
\begin{proof}[Proof of Lemma~\ref{le:10}]
For $x,y\in \RR$,
$
f(x)  = f(y) +  \int_y^x f'
$
and thus by the Cauchy-Schwarz inequality
\[
f^2(x) \lesssim f^2(y)   +  (|x|+|y|) \int (f')^2.
\]
Multiplying by $\rho(y)$ and integrating in $y$, then 
multiplying by $\rho^{2\kappa}(x)$ and taking the supremum in $x$, we obtain the
first inequality.

For the second inequality, we fix $A\geq 1$.
For $0\leq y\leq 1\leq x\leq A$, we write
\[
f^2(x) = f^2(y) + 2 \int_y^x f'(z)f(z) dz
\]
so that multiplying by $\rho^\kappa(x) $ and integrating in $x\in [1,A]$,
\begin{align*}
\int_1^A f^2 \rho^\kappa & \lesssim f^2(y) + \int_1^A \left(\int_0^x |f'(z)| |f(z)| dz \right) \rho^\kappa(x) dx\\
& \lesssim f^2(y) + \int_0^A |f'(z)| |f(z)| \left(\int_z^A \rho^\kappa(x) dx\right) dz 
\lesssim f^2(y) + \int_0^A |f'| |f| \rho^\kappa.
\end{align*}
Thus, by the Cauchy-Schwarz inequality,
\[
\int_1^A f^2 \rho^\kappa \lesssim f^2(y) 
+ \left(\int_0^A (f')^2\right)^\frac12 \left(\int_0^A f^2 \rho^{2\kappa}\right)^\frac 12.
\]
Integrating in $y\in [0,1]$ yields the estimate for $\int_1^A f^2 \rho^\kappa$.
The estimate for $\int_0^1 f^2 \rho^\kappa$ is clear.
\end{proof}
 
Now, we define 
\[
\cN = \left(\int w_1^2+\int (\partial_x w_2)^2+ \int (\partial_x^2 w_2)^2\right)^\frac12.
\]
\begin{lemma}\label{le:04}
For any $\kappa>0$,
\[
\int (\partial_x \ep_2)^2 + \int \ep_2^2 \rho^\kappa 
+\|\ep_2^2 \rho^{\kappa}\|_{L^\infty}
+ \int w_2^2 \rho^\kappa \lesssim \cN^2.
\]
Moreover,
\[
|\partial_x w_2(x)|+|\partial_x^2 w_2 (x)|\lesssim \sup_{|y|\geq |x|} |\partial_x \ep_2(y)| 
+ Q_c(x) |\ep_2(x)|\lesssim \rho^\gamma.
\]
\end{lemma}

\begin{lemma}\label{le:05}
For any $\kappa>0$,
\[
\int (\partial_x \ep_1)^2 \rho^\kappa + \int \ep_1^2 \rho^\kappa
+\|\ep_1^2 \rho^{\kappa}\|_{L^\infty}\lesssim \cN^2.
\]
Moreover,
\[
|w_1(x)|\lesssim |\partial_x \ep_1(x)|
+ \sup_{|y|\geq |x|} |\partial_x \ep_2(y)| + Q_c(x) |\ep_2(x)|+
|\zeta(x)|\lesssim \rho^\gamma.
\]
\end{lemma}
\begin{lemma}\label{le:06}
It holds
\[
|\dot a-c|^2+|\dot c|+|\dot \theta|^2\lesssim \cN^2, \quad
\|\rho^\frac\gamma4\partial_x F_1\|\lesssim \cN^2.
\]
Moreover, $F_2=F_{2,1}+\partial_x F_{2,2}$ where
\[
\|\rho^\frac\gamma4 F_{2,1}\|+\|\rho^\frac\gamma4 F_{2,2}\|
\lesssim \cN^2.
\]
\end{lemma}
\begin{proof}[Proof of Lemma~\ref{le:04}]
The orthogonality $\int \ep_2 Q_c=0$ was already used to construct $w_2$ in \S~\ref{S:3.2}.
Now, we use the second orthogonality relation on $\ep_2$, which is $\int \ep_2 R_c Q_c=0$.
Indeed, 
using $Q_c \cdot S_c^\star = - \partial_x \cdot Q_c$, it implies that
\[
0=\int\ep_2 R_c Q_c=\int (S_c^\star w_2) R_c Q_c
=-\int R_c \partial_x ( w_2 Q_c) = \frac1{\sqrt{2}} \int Q_c^2 w_2.
\]
From this orthogonality relation, it is standard to prove
the inequality
\[
\int w_2^2 \rho^\kappa\lesssim \int (\partial_x w_2)^2.
\]
From $\ep_2=S_c^\star w_2$, we obtain
$|\ep_2|\lesssim |\partial_x w_2| + |w_2|$
and so
\[
\int \ep_2^2 \rho^\kappa \lesssim \int \left((\partial_x w_2)^2 + w_2^2\right) \rho^\kappa \lesssim \int (\partial_x w_2)^2.
\]
Moreover, differentiating, we have $\partial_x \ep_2 = -\partial_x^2 w_2
+ \sqrt{2} R_c \partial_x w_2 + Q_c w_2$, and so
\[
\int (\partial_x \ep_2)^2 \lesssim \int (\partial_x^2 w_2)^2 + \int (\partial_x w_2)^2 
+ \int w_2^2 \rho^\kappa.
\]
It is standard to obtain the $L^\infty$ bound from the above.

The second estimate of the lemma follows from  Lemma~\ref{le:03}
and Lemma~\ref{le:01}.
\end{proof}
\begin{proof}[Proof of Lemma~\ref{le:05}]
By the definition of $w_1$ and $v_1=S_c \ep_1$, we have
\[
S_c \ep_1 = Q_c \partial_x(\ep_1/Q_c) = -\frac1{\sqrt{2}} \ep_1^2 + h \quad\mbox{where}\quad
h=w_1-cS_cw_2 - \frac1{\sqrt{2}} \ep_2^2.
\]
By integration on $[0,x]$, we have
\[
\ep_1 = b Q_c + Q_c \int_0^x \frac{\ep_1^2}{Q_c} + Q_c\int_0^x \frac{h}{Q_c}
\]
for some integration constant $b$.
Using the orthogonality relation $\int \ep_1 Q_c=0$, we obtain
\[
b \int Q_c^2 
= -\int Q_c^2 \int_0^x \frac{\ep_1^2}{Q_c} - \int Q_c^2 \int_0^x \frac{h}{Q_c} 
\]
and so by the Fubini theorem and the Cauchy-Schwarz inequality
\[
|b| \lesssim \int \ep_1^2 Q_c + \left(\int h^2 Q_c^\frac 32\right)^\frac12.
\]
Using Lemma~\ref{le:10}, since $\int (\partial_x\ep_1)^2\lesssim \alpha_0$
and $\int_0^1 \ep_1^2 \lesssim \int \ep_1^2 Q_c^2\lesssim \alpha_0$ (Lemma \ref{le:01}), we obtain
\[
b^2 \lesssim \alpha_0 \int \ep_1^2 Q_c^2 +\int h^2 Q_c^\frac 32.
\]
Then we estimate
\[
\int \ep_1^2 \rho^\kappa \lesssim b^2 \int \rho^\kappa Q_c^2 +\int \rho^\kappa Q_c^2 \left(\int_0^x \frac{\ep_1^2}{Q_c} \right)^2
+ \int \rho^\kappa Q_c^2\left(\int_0^x \frac{h}{Q_c}\right)^2.
\]
For the second term on the right-hand side, we use
Lemma~\ref{le:10} to obtain
\begin{multline*}
\int \rho^\kappa Q_c^2 \left(\int_0^x \frac{\ep_1^2}{Q_c} \right)^2
= \int \rho^\kappa Q_c^2 \left(\int_0^x \frac{\ep_1^2\rho^\frac\kappa4}
{Q_c\rho^\frac\kappa4} \right)^2
 \leq \int \rho^\frac\kappa2 \left(\int_0^x \ep_1^2 \rho^\frac\kappa4 \right)^2\\
\lesssim \left(\int \ep_1^2 \rho^\kappa\right)^2 
+\left(\int (\partial_x \ep_1)^2\right) \int \rho^\frac\kappa2 \left|\int_0^x \ep_1^2\rho^\frac\kappa2\right|
\lesssim \alpha_0 \int \ep_1^2 \rho^\kappa.
\end{multline*}
Besides,
\[
 \int \rho^\kappa Q_c^2\left(\int_0^x \frac{h}{Q_c}\right)^2
 \lesssim \int h^2 \rho^{\frac\kappa2}.
\]
Combining the previous estimates, assuming $\kappa\leq \beta/2$, we have obtained
\[
\int \ep_1^2 \rho^\kappa \lesssim \alpha_0 \int \ep_1^2 \rho^\kappa +\int h^2 \rho^{\frac\kappa2}
\]
and thus for $\alpha_0$ small enough, we have proved
$\int \ep_1^2 \rho^\kappa\lesssim \int h^2 \rho^{\frac\kappa2}$.
Using the expression of $h$,
\[
\int \ep_1^2 \rho^\kappa \lesssim \int w_1^2+\int (\partial_x w_2)^2 +\int w_2^2 \rho^{\frac\kappa2}
+ \int \ep_2^4 \rho^{\frac\kappa2}.
\]
Using Lemma~\ref{le:10}, we have
$\|\ep_2 \rho^{\kappa/8}\|_{L^\infty}\lesssim \|\ep_2\|_{\cH}\lesssim 1$
and so by the proof of Lemma~\ref{le:04}, we obtain
\[
\int \ep_1^2 \rho^\kappa \lesssim \int w_1^2+\int (\partial_x w_2)^2.
\]
Using
$
\partial_x \ep_1
= -w_1+2c\partial_x w_2+ \sqrt{2}R_c \ep_1+ c \ep_2+\frac{1}{\sqrt{2}} |\ep|^2 
$
we obtain
\[
\int (\partial_x \ep_1)^2 \rho^\kappa \lesssim \int w_1^2+\int (\partial_x w_2)^2.
\]
To prove the second estimate, we recall that$
w_1=\partial_x \ep_1+2c\partial_x w_2+\frac{\zeta}{\sqrt{2}}
$
so that 
\[
|w_1(x)|\lesssim |\partial_x \ep_1(x)|
+ |\partial_x w_2(x)| + Q_c(x) |\ep_2(x)|+
|\zeta(x)|.
\]
We finish the proof using 
Lemma~\ref{le:01} and Lemma~\ref{le:04}.
\end{proof}
\begin{proof}[Proof Lemma~\ref{le:06}]
From Lemma~\ref{le:01} and Lemmas~\ref{le:04} and
\ref{le:05}, one has
\[
|\dot a-c|^2+|\dot c|+|\dot \theta|^2
\lesssim \int |\ep|^2 \rho^\gamma\lesssim \cN^2.
\]

\emph{Estimate of $\partial_x F_1$.}
By the definition of $F_1$
\[
\partial_x F_1 = \partial_x P_1 + \sqrt{2} (\partial_x \ep_1)\ep_1
+\sqrt{2} (\partial_x \ep_2)\ep_2 + \partial_x \Theta_1.
\]
By the definition of $P_1$ and Lemma~\ref{le:03},
\[
\|\rho^\frac\gamma4\partial_x P_1\|
\lesssim \|\rho^\frac\gamma8W^\perp\|+ \|\rho^\frac\gamma8N_1^\perp\|
\lesssim \|\rho^\frac\gamma8W\|+ \|\rho^\frac\gamma8N_1\|.
\]
From Lemma~\ref{le:04}, one has $\|\rho^\frac\gamma8 w_2\|\lesssim \cN$, and so
by the definition of $W$, $|\partial_c R_c|\lesssim 1$ and $|\dot c|\lesssim \cN^2$, it holds
$\|\rho^\frac\gamma8W\|\lesssim\cN^3$.
Moreover, $|N_1|\lesssim|\ep|^2$ and thus from Lemmas~\ref{le:04} and
\ref{le:05}, it holds
\[
\|\rho^\frac\gamma8N_1\|\lesssim \|\rho^\frac\gamma{16}\ep\|
\|\rho^\frac\gamma{16}\ep\|_{L^\infty}\lesssim \cN^2.
\]
Therefore, $\|\rho^\frac\gamma4\partial_x P_1\|\lesssim \cN^2$ is proved.
Similarly, we see that
\[
\|\rho^\frac\gamma4(\partial_x \ep_1)\ep_1\|
+\|\rho^\frac\gamma4(\partial_x \ep_2)\ep_2 \|
\lesssim \cN^2.
\]
Now, we deal with $\partial_x \Theta_1$.
We decompose
$\Omega_1^\perp=\Omega_{1,1}^\perp+\Omega_{1,2}^\perp
=-S_c^\star \Theta_{1,1}-S_c^\star\Theta_{1,2}$
where
\[
\Omega_{1,1}=-(\dot a -c)\partial_x\ep_2+\frac{\dot c}{\sqrt{2}}
+\dot\theta \ep_2 ,\quad
\Omega_{1,2}=\Omega_{1,2}^\perp=\dot\theta R_c.
\]
The term $\Theta_{1,2}$ could be problematic since $\Omega_{1,2}=\dot \theta R_c$ 
 is linear in $\cN$.
However, since $S_c^\star 1 = \sqrt{2}R_c$, we have $\Theta_{1,2}=-\frac1{\sqrt2}\dot \theta$.
Thus, $\partial_x \Theta_{1,2}=0$ and this term actually has no contribution to
$\partial_x F_1$. The terms in $\Omega_{1,1}$ are quadratic and as before, by Lemma~\ref{le:03}, we have
\[
\|\rho^\frac\gamma4\partial_x \Theta_{1,1}\|
\lesssim \|\rho^\frac\gamma8\Omega_{1,1}^\perp\|
\lesssim \|\rho^\frac\gamma8\Omega_{1,1}\|\lesssim \cN^2.
\]

\emph{Estimate of $F_2$.}
In the definition of $F_2$, we replace $P_2$ by its expression
and we insert the expressions of $\partial_t \ep_1$ and $\partial_t\ep_2$
\begin{align*}
F_2 &= S_c N_2+\sqrt{2} \dot c \ep_1 \partial_c R_c+\Theta_2+\dot c S_c w_2+\sqrt{2} c\dot c w_2\partial_c R_c+c S_c P_1+c S_c \Theta_1\\
&\quad +2 \sqrt{2} \ep_1 (-\partial_x^2\ep_2+c\partial_x\ep_1-Q_c\ep_2
+\frac{c}{\sqrt{2}}\zeta-\zeta\ep_2+\Omega_2)\\
&\quad
+2\sqrt{2}\ep_2(\partial_x^2\ep_1+c\partial_x\ep_2+Q_c\ep_1-R_c\zeta-\zeta\ep_1-\Omega_1).
\end{align*}
Then, replacing $\Theta_2$ and $\Theta_{1,2}$ by their definitions, 
we split $F_2 =F_{2,1}+\partial_x F_{2,2}$
where
\begin{align*}
F_{2,1}&=S_c N_2+\sqrt{2} \dot c \ep_1 \partial_c R_c+\dot c S_c w_2+\sqrt{2} c\dot c w_2\partial_c R_c+c S_c P_1\\
&\quad -(\dot a-c)\sqrt{2}R_c\partial_x\ep_1-\dot cS_c\partial_cR_c
+\dot \theta S_c \ep_2
+c S_c \Theta_{1,1}\\
&\quad +\ep_1(2c\zeta+2\sqrt{2}c\zeta\ep_2+2\sqrt{2}\Omega_2)
-2\sqrt{2}\ep_2(R_c\zeta+\zeta\ep_1+\Omega_1),\\
F_{2,2}&=
-(\dot a-c)\partial_x\ep_1
+2\sqrt{2}(\ep_2\partial_x\ep_1-\ep_1\partial_x\ep_2+c\ep_1\ep_2).
\end{align*}
Observe that for the terms
$\Theta_2$ and $cS_c\Theta_1$, we have used $S_cR_c'=0$ and a cancellation of two terms in $\dot \theta$,
exactly as for $F_1$.
From the expressions of $F_{2,1}$ and $F_{2,2}$ above, in which all the terms 
are quadratic, we obtain similarly as before
$
\|\rho^\frac\gamma4 F_{2,1}\|+\|\rho^\frac\gamma4F_{2,2}\|\lesssim \cN^2
$.
\end{proof}

\section{Virial identity for the transformed problem}
Set 
\[
\cI=\int x (\partial_x w_2)(w_1-c\partial_x w_2)
\]
Note that by Lemmas~\ref{le:04} and~\ref{le:05}, the functional $\cI$ is well defined
and uniformly bounded.

\begin{lemma}\label{le:07}
It holds $\dot{\cI}=\frac12\cQ+ \cR$ where
\begin{align*}
\cQ &= \int (w_1-2c\partial_x w_2)^2+\beta^2\int (\partial_x w_2)^2
+3 \int (\partial_x^2 w_2)^2\\
\cR&=\int x(\partial_x w_2) F_2+\int x (\partial_x F_1) (w_1-2 c \partial_x w_2)-\dot c \int x(\partial_x w_2)^2.
\end{align*}
\end{lemma}
\begin{proof}
We compute from the system for $(w_1,w_2)$
\begin{align*}
\dot{\cI}&=\int x(\partial_{x}\partial_t w_2)(w_1-c\partial_x w_2)+\int x(\partial_x w_2) (\partial_t w_1 -c \partial_{x}\partial_t w_2)-\dot c\int x (\partial_x w_2)^2\\
&=\int x(\partial_x(-w_1+2 c\partial_x w_2+F_1))(w_1-c\partial_x w_2)
+\int x (\partial_x w_2)((\partial_{x}^2-\beta^2)\partial_{x}^2 w_2+F_2)\\
&\quad -c\int x (\partial_x w_2)\partial_x(-w_1+2 c\partial_x w_2 + F_1)-\dot c\int x (\partial_x w_2)^2 = \cQ+\cR
\end{align*}
where $\cR$ is defined as in the statement of the lemma and 
\begin{align*}
\cQ &=-2\int x(\partial_x(w_1-2 c\partial_x w_2)) (w_1-2c\partial_x w_2)
+2\int x(\partial_x w_2)(\partial_{x}^2-\beta^2)\partial_{x}^2 w_2\\
&= \int (w_1-2c\partial_x w_2)^2+\beta^2\int (\partial_x w_2)^2+3 \int (\partial_x^2 w_2)^2
\end{align*}
using integrations by parts.
\end{proof}

\begin{lemma}\label{le:08}
It holds
\[
\cQ \geq \frac {\beta^2}{16} \cN^2.
\]
\end{lemma}
\begin{proof}
Let $b=\sqrt{1+7c^2/2}$.
We rewrite
\begin{equation*}
\cQ = \frac{\beta^2}{2b} \int w_1^2 + \frac{\beta^2}{2} \int (\partial_x w_2)^2
+ 3\int (\partial_x^2 w_2)^2
 + \int \left( \frac{2c}{b} w_1 - b \partial_x w_2\right)^2
\end{equation*}
which is sufficient to prove the result.
\end{proof}
\section{Proof of the Liouville theorem}

\begin{lemma}\label{le:09}
It holds
\[
|\cR| \lesssim \cN^3.
\]
\end{lemma}
\begin{remark}
Formally, $\cR$ contains quadratic terms and should satisfy an estimate of the form
$|\cR|\lesssim \cN^4$. However, some loss is necessary to recover space decay by
Lemmas~\ref{le:04} and~\ref{le:05}.
\end{remark}
\begin{proof}
We set $\cR=\cR_1+\cR_2+\cR_3$, where
\[
\cR_1=\int x(\partial_x w_2) F_2,\quad
\cR_2=\int x (\partial_x F_1) (w_1-2 c \partial_x w_2),\quad
\cR_3=-\dot c \int x(\partial_x w_2)^2.
\] 
\emph{Estimate of $\cR_1$.}
We use the decomposition $F_2=F_{2,1}+\partial_x F_{2,2}$ from Lemma~\ref{le:06} to write
\[
\cR_1 = \int x(\partial_x w_2)F_{2,1}-\int (\partial_x w_2)F_{2,2}-\int x(\partial_x^2 w_2)F_{2,2}.
\]
Thus, by the Cauchy-Schwarz inequality
\[
|\cR_1|\lesssim
\|\rho^\frac\gamma4 F_{2,1}\| \|x\rho^{-\frac\gamma4} \partial_xw_2\|
+\|\rho^\frac\gamma4 F_{2,2}\| \left(\|\rho^{-\frac\gamma4}\partial_xw_2\|
+\|x\rho^{-\frac\gamma4} \partial_x^2w_2\|\right)
\]
By the Cauchy-Schwarz inequality and then Lemma~\ref{le:04}
\[
\left|\int x^2\rho^{-\frac\gamma2}(\partial_x w_2)^2\right|
\lesssim \cN \left(\int x^4\rho^{-\gamma} (\partial_x w_2)^2\right)^\frac 12
\lesssim \cN \left(\int x^2 \rho^{\gamma}\right)^\frac 12 
\lesssim \cN 
\]
and similarly
$\|\rho^{-\frac\gamma4}\partial_xw_2\|
+\|x\rho^{-\frac\gamma4} \partial_x^2w_2\|\lesssim \cN^\frac 12$.
Thus, the estimate for $\cR_1$ follows from Lemma~\ref{le:06}.

\emph{Estimate of $\cR_2$.} We have
\[
|\cR_2|\lesssim \|\rho^\frac\gamma4\partial_x F_1\|
\big( \|\rho^{-\frac\gamma4}xw_1\|+\|\rho^{-\frac\gamma4}x\partial_xw_2\|\big).
\]
As before, using Lemmas~\ref{le:04},~\ref{le:05},
$\|\rho^{-\frac\gamma4}xw_1\|+\|\rho^{-\frac\gamma4}x\partial_xw_2\|\lesssim \cN^\frac 12$.
Thus, Lemma~\ref{le:06} implies the result for $\cR_2$.

\emph{Estimate of $\cR_3$.}
By Lemma~\ref{le:04},
$\int |x|(\partial_x w_2)^2
\lesssim \cN$
and the estimate of $\cR_3$ follows from Lemma~\ref{le:06}.
\end{proof}
From Lemmas~\ref{le:07}, \ref{le:08} and~\ref{le:09},
for $\alpha_0$ sufficiently small, which implies that $\cN$ is small,
it holds
$\dot\cI \gtrsim \cN^2$.
Thus, by integration in $t\in (-\infty,+\infty)$, and the bound on $\cI$, we have $\int_{-\infty}^{+\infty} \cN^2 <+\infty$.
Thus, there exists a sequence $t_n\to+\infty$, such that
$\lim_{n\to+\infty} \cN(t_n)=0$.
From Lemmas~\ref{le:04} and~\ref{le:05}, we obtain
$\lim_{n\to+\infty} \|\ep(t_n)\|_\cH=0$.
Using also the pointwise decay estimate of $\eta$ in Lemma~\ref{le:01},
we have $\lim_{n\to+\infty} \|\eta(t_n)-Q_{c(t_n)}\|=0$.
This implies that 
\[
\lim_{n\to+\infty}d(\psi(t_n),e^{\ii \theta(t_n)} U_{c(t_n)}(x-a(t_n)))=0.
\]
By the stability statement, we obtain
that $\psi$ is exactly a soliton.
 
\end{document}